%% file: z.tex
\pdfoutput=1
\documentclass{amsart}
\usepackage{latexsym}
\usepackage{amsfonts}
\input{preamble}

\begin{document}
\title{Reduction theorems for the 
Strong Real {J}acobian Conjecture}
\author{L. Andrew Campbell} 
\address{908 Fire Dance Lane \\
Palm Desert CA 92211 \\ USA}
\email{lacamp@alum.mit.edu}

\input{abstract}

\maketitle

\section{Introduction}\label{intro}

\input{intro}

\section{Background}\label{history}
\input{history}

\section{Stable and Segre equivalence}\label{equiv}

\input{equiv}

\section{Gorni-Zampieri pairing}\label{gorzamp}

\input{gorzamp}

\section{Main results}\label{final}

\input{final}

\section{Acknowledgments}\label{ack}

\input{ack}

\input{zbib}
\end{document}

%% file: preamble.tex
\newcommand{\C}{\mathbb{C}}
\newcommand{\R}{\mathbb{R}}

\newcommand{\To}{\rightarrow}

\newcommand{\comps}{{f_1,\ldots,f_n}}
\newcommand{\vars}{{x_1,\ldots,x_n}}

\theoremstyle{plain}
\newtheorem{Thm}{Theorem}

\newtheorem{Conj*}{Conjecture}

\theoremstyle{remark}

\newcommand{\J}{{J}acobian\ }
\newcommand{\p}{polynomial\ }
\newcommand{\pk}{{P}inchuk\ }
\newcommand{\nz}{nonzero\ }

\newcommand{\Druz}{Dru{\.z}kowski\ }

\newcommand{\F}{F: \R^n \To \R^n}
\newcommand{\G}{G: \R^m \To \R^m}

\newcommand{\edr}{everywhere defined rational\ }

%% file: abstract.tex
\begin{abstract}
Implementations of known reductions of
the Strong Real Jacobian Conjecture (SRJC),
to the case of an identity map plus 
cubic homogeneous or cubic linear terms, 
and to the case of 
gradient maps, are shown to preserve significant 
algebraic and geometric properties of the maps involved. 
That permits the separate formulation and reduction, 
though not so far the solution, of the SRJC for classes of 
nonsingular polynomial 
endomorphisms of real n-space 
that exclude the Pinchuk counterexamples 
to the SRJC, for instance those that induce rational 
function field extensions of a given fixed odd degree. 
\end{abstract}

%% file: intro.tex
The \J Conjecture (JC) \cite{BCW82,ArnoBook} 
asserts that a polynomial map $F: k^n \To k^n$,
where $k$ is a field of characteristic zero, has a polynomial inverse if  
it is a Keller map \cite{Keller}, which means that its \J determinant, $j(F)$, is a \nz element of $k$.
The JC is still not settled for any $n > 1$ and any specific field $k$ of characteristic zero. 
It is well known that it would suffice to prove the JC for
 $k=\R$ and all  positive $n$. 
There are many  generalizations to endomorphisms of $\R^n$
\cite{RealCase,GlobU}. 
The most natural is the Strong Real \J Conjecture (SRJC), which 
asserts that a polynomial map $F: \R^n \To \R^n$,
has a real analytic  inverse if it is nonsingular, meaning that
$j(F)$, whether constant or not, vanishes nowhere on $\R^n$. 
However, Sergey Pinchuk exhibited a family 
of counterexamples for $n=2$ \cite{Pinchuk}. 
They are also counterexamples to the 
 Rational Real \J Conjecture (RRJC)  \cite{RRJC}, which
  is the extension of the SRJC to include everywhere 
defined rational nonsingular endomorphisms.
Everywhere defined means that each component of the map 
can be expressed as the quotient of two polynomials 
with a nowhere vanishing denominator. Any such $\F$ has finite fibers of size at most the degree of the
 associated  finite algebraic extension of rational function fields.

Let dex denote that extension degree, mfs the maximum fiber size, 
and sag the size of the automorphism group of  
the extension.  
While dex and mfs are not generally equal, 
they are always of the same parity. 
The conditions dex odd, mfs odd, and sag = 1 are all 
necessary for invertibility, and mfs = 1 or dex = 1 
is sufficient.  
All the Pinchuk counterexamples satisfy dex = 6, mfs = 2, and 
sag = 1 
 \cite{aspc,PMFF}. 
Thus the simplest unproved and unrefuted  challenge conjecture in this arena is that dex = 3 and sag = 1 is sufficient in the \p case.

Two known reduction procedures for the SRJC, 
to the case of maps
of cubic homogeneous, 
or even better cubic linear,  type 
\cite{Jagzhev,BCW82,Effective,ArnoBook} 
 and to the case of a symmetric \J matrix \cite{Meng},  
are shown here to preserve the numerical attributes 
dex, mfs, and sag.  

In consequence, the two reductions 
 can be applied to conjectures involving those attributes, such
as the above challenge conjecture.

%% file: history.tex
There are two classic reductions of the ordinary JC 
to Yagzhev maps \cite{Jagzhev,BCW82} and to \Druz maps \cite{Effective}.
A Yagzhev map is a polynomial map of the form 
$F=X+H$, where $X=(\vars)$, and each 
component of $H$ is a cubic homogeneous polynomial in
the variables $\vars$. 
Yagzhev maps are also called maps of cubic homogeneous type. 
A \Druz map (or map of cubic linear type)
is a Yagzhev map, 
for which the components of $H$ are cubes of linear 
forms ($h_i=l_i^3$). In a departure from the convention in 
some other works, 
these definitions impose no restriction on $j(F)$, 
beyond the obvious $j(F)(0)=1$. 
Note, however, that a Yagzhev map $F=X+H$ is a Keller map 
if, and only if, $H$ has a \J matrix, $J(H)$, that  
 is nilpotent, since both assertions are 
just different ways of saying that the formal power series 
matrix inverse of $J(F)$ is polynomial.

Reduction theorem proofs use the strategy of 
transforming an original map into a map 
of the desired form in a succession of steps that preserve 
the truth value of certain key properties (and typically 
increase the number of variables). 

For the JC, $\C$ is usually selected as the ground field and 
the key properties are the Keller property 
and the existence of a polynomial inverse.
Such proofs  then apply over  
any ground field of  
 characteristic zero, including $\R$.
But the strategy and specific steps can be applied more 
generally than just to polynomial Keller maps and yields, 
for instance, a reduction of the SRJC to the cubic linear 
case \cite{Effective}. 
\Druz noted this explicitly, 
with the preserved properties being a nowhere vanishing \J determinant and  bijectivity.

Historical Note. 
At the 1997 conference in Lincoln, Nebraska, to honor 
the mathematical work of Gary H. Meisters, it was suggested 
by T. Parthasarathy that the SRJC reduction be attempted 
for the 1994 counterexample of Pinchuk. 
The challenge was taken up by Engelbert Hubbers, and 
in 1999 he demonstrated the existence of a 
counterexample to the SRJC of cubic linear type, 
coincidentally in dimension 1999. 
He started with exactly the specific Pinchuk map of 
total degree 25  circulated by Arno van den Essen in June 1994, 
which can be found in 
\cite{ArnoBook}. 
He then used a computer algebra system
to verify a human guided reduction path to a Yagzhev map
in dimension 203, then explicitly computed a
Gorni-Zampieri pairing \cite{GZpairs} 
to a \Druz map in dimension 1999, using sparse matrix 
representations as necessary.
These details are excerpted from a comprehensive 
unpublished note by Hubbers, which he made available. 

Remark. 
\Druz obviously did not use GZ pairing, since it
was unknown at the time.  
But it also preserves the same two key properties 
in the SRJC context.

More recently, reductions of the ordinary JC to the 
symmetric case have been considered, primarily over 
$\R$ and $\C$. 
Let $k$ denote a field of characteristic zero. 
In the JC world  a polynomial map 
$F:k^n \To k^n$ is often called symmetric, in a 
startling abuse of language, if $J(F)$ is a symmetric 
matrix. In that case, $F$ is the gradient map of a polynomial 
function $h:k^n \To k$ and $J(F)$ is the Hessian matrix of second 
order partial derivatives of $h$. So in the symmetric case, 
the JC becomes the Hessian conjecture (HC), namely that gradient 
maps of polynomials with constant nonzero Hessian 
determinant have polynomial inverses. 
In \cite{Meng}, Guowu Meng proves 
 the equivalence of the JC and the HC, using what 
he refers to as a trick. Meng's trick replaces a map 
$F=(\comps)$ in the variables $\vars$ by the map  
in the $2n$ variables $y_1,\ldots,y_n,\vars$ 
obtained by taking the gradient of the 
scalar function $y_1f_1 + \cdots + y_nf_n$. 
For $k=\R$, this construction works even for 
twice continuously differentiable maps. 
In the SRJC context it provides a one step reduction to the 
symmetric case that 
also preserves the Keller property 
in both directions. 

In \cite{SymmetricCase}, Michiel 
de Bondt and Arno van den Essen prove a more targeted 
reduction over $\C$, namely to symmetric Keller Yagzhev maps. 
The reduction process involves the use of $\sqrt{-1}$, and 
if applied to a real Keller map may yield a Yagzhev map 
that is not real. Interestingly, it has been shown 

that all complex symmetric Keller \Druz maps 
have polynomial inverses \cite{SymmetricDmaps,DruSymDru}.

%% file: equiv.tex
Two maps, $F$ and $G$, from a topological space $A$ to 
another one $B$, are called topologically equivalent 
if $F = h_B \circ G \circ h_A$, where $h_A$ and $h_B$ are 
homeomorphisms, respectively of 
$A$ to itself and of $B$ to itself. 
In other words, $F$ and $G$ are the same map up to 
coordinate changes in the domain and codomain by 
topological automorphisms. 
Topological stable equivalence for the set of all maps 
$F:\R^n \To \R^n$ in all dimensions $n>0$ is the 
equivalence relation generated by (1) topological 
equivalences, and (2) the equivalence of any map 
$F=(\comps)$, and its extension 
by fresh variables to  
$G=(f_1,\ldots,f_n,x_{n+1},\ldots,x_m)$
for any $m>n$.  
There are many other types of stable equivalence, such as 
real analytic or polynomial, each characterized by the type 
of automorphisms allowed for (global) coordinate changes. 
Stable equivalence, unqualified, will refer to the 
least restrictive, purely set theoretic, type, 
with all bijections allowed as automorphisms.

For brevity, call $F:\R^n \To \R^n$ (1) nondegenerate if 
$j(F)$ is not identically zero, 
(2) nonsingular if $j(F)\ne 0$ everywhere, and (3) a Keller 
map if $j(F)$ is a nonzero constant. 
These terms are meant to imply that $J(F)$, the  \J matrix 
of $F$, exists at 
every point of $\R^n$, and can be applied to any such 
$F$ if the corresponding restriction on 
$j(F)$ is satisfied.
For polynomial stable equivalence, the applicable 
automorphisms are \p maps with \p inverses, making 
it obvious that such equivalence preserves each of the
above three properties. All preservation properties
in this paper apply equally well in both directions. 
It also clearly preserves each of the properties
of being everywhere defined rational, polynomial,
injective, surjective, or bijective.
If the maps are nonsingular, it preserves the 
existence of a rational or polynomial inverse.

To verify this last assertion in the case of extension
by fresh variables, one checks the \J matrices to see
that the appropriate part of an inverse in the larger
number of variables is independent of the fresh
variables and restricts to an inverse in the smaller
number of variables.

The slightly more general and less familiar 
concept of birational stable equivalence allows the use of automorphisms that are 
\edr maps with \edr inverses.  
By inspection of the arguments in the polynomial case, 
one sees easily that birational stable equivalence 
has all the preservation properties listed above for the 
\p case, except that it need not preserve 
\p maps, \p inverses, or the Keller property. 

If $F=(\comps)$ is an everywhere defined rational map 
and is nondegenerate, then its components are algebraically 
independent 
over $\R$ in the field $\R(X)$ of rational 
functions in the coordinate variables $X=\vars$, 
and so they generate a subfield $\R(F) \subseteq \R(X)$  over $\R$, 
that is also a rational function field in $n$ variables over $\R$. 
Even without nonsingularity, the extension $\R(X)/\R(F)$  
permits the  definition of dex and sag as in  the introduction. 
The extension degree $d=\text{dex}$ is finite and equal to the 
degree of the minimal polynomial over $\R(F)$ of any $h \in \R(X)$ 
that is primitive, meaning that $h$ generates $\R(X)$ as a field over $\R(F)$. 
For such an $h$,  the powers $h^i$ for $i=0,\ldots,d-1$ are a basis 
for $\R(X)$ as a vector space over $\R(F)$. 
An automorphism of the extension is, by definition, a field automorphism 
of $\R(X)$ that fixes every element of $\R(F)$.  
So it is linear over $\R(F)$, and a multiplicative homomorphism, 
hence completely determined by its value on $h$.  
That value must be a root of the minimal polynomial  of $h$ and 
must lie in $\R(X)$, and any such root determines a unique 
automorphism of the extension.  
Thus sag is the number of such roots, which is therefore the same for any choice of $h$.

In another relaxation of assumptions, it suffices 
to assume that $F$ is an everywhere defined  nondegenerate 
rational map and an open map in order 
to conclude that it is quasifinite and that the 
maximum fiber size is at most dex. 

The main concern here is the case of  \edr nonsingular maps, 
for which all the assertions in the 
introduction are proved in \cite{RRJC}. 

\begin{Thm}
Assume $\F$ and $\G$ are birationally  stably 
equivalent. If either one is an everywhere defined 
rational nonsingular 
map, then so is the other, and each of the numerical
attributes dex, sag, and mfs has the same value
for both maps. 
\end{Thm}

\begin{proof}
Only the equality of the numerical attributes needs checking.
And it needs to be checked only for the generating
equivalences.

Suppose first that $m>n$ and $G=(F,Z)$, with
$Z$ a list of fresh variables. 
For any $y \in \R^n$ and $z \in \R^{m-n}$, the 
fiber of $F$ over $y$ is the same size as the fiber of
$G$ over $(y,z)$, so mfs is preserved. 
A primitive element for $\R(X)/\R(F)$ is clearly  
also a primitive element for $\R(X,Z)/\R(G)$. 
Since the tensor product over $\R$ with $\R(Z)$ is an exact functor, 
the associated power basis for the first extension 
is also one for the second, and so dex is preserved. 
An automorphism of the extension $\R(X,Z)/\R(G)$ 
is determined by the image of the primitive element. 
That image is a root of the minimal polynomial 
for the primitive element. That polynomial has coefficients 
independent of the fresh variables in $Z$.  
On a Zariski open subset of $\R^m$, where the root is 
a real analytic function of the coefficients, the root is 
independent of the fresh variables. 
So the first order partials with respect to those variables 
are identically zero and the root lies in $\R(X)$. 
Thus the automorphism is uniquely the natural lift  of an 
automorphism of $\R(X)/\R(F)$, and so sag is preserved. 

Second, suppose that  $m=n$ and  $G=A \circ F \circ B$, 
with $A$ and $B$  
everywhere defined birational automorphisms.  
Viewing them as coordinate changes makes it clear that 
 mfs is preserved.  
It suffices to consider further only the 
special cases i) $G=A \circ F$ and ii) $G=F \circ B$, 
and $A$, $B$, and their inverses do not need 
to be defined everywhere. 

In case i) $A$ induces an automorphism of $\R(F)$, so $\R(F)$ and $\R(A \circ F)$ are the same subfield of $\R(X)$, hence the two extensions are the same and so have the same properties. 
In case ii) the two extensions are $\R(X)/\R(F)$ and 
 $\R(X)/\R(F \circ B)$, which are generally different extensions.
Take a primitive element $h$ for the first extension,
then apply the automorphism of $\R(X)$ induced by 
$B$ to $h$, its minimal polynomial over $\R(F)$,
and the roots of that polynomial in $\R(X)$.  The image of $h$
is primitive over $\R(F \circ B)$, the new polynomial
is  
irreducible there, and 
the roots of the two polynomials correspond.
So dex and sag are preserved. 
\end{proof}

Let $\F$ be a \p map satisfying $F(0)=0$. 
Then $H(x,t)=(1/t)F(tx)$ is \p and provides, 
for $0 \le t \le 1$, a homotopy 
between the linear part of $F$ and $F$ itself.

This Segre homotopy \cite{Segre} can be generalized in many ways, 
e.g. to the case of a complex map or parameter $t$ and to 
analytic or rational maps, not to 
mention formal and convergent power series. 
It is used here to define the concept  of Segre equivalence 
on the set of 
real analytic endomorphisms of $\R^n$ ($n>0$) that fix $0$. 
  It is the equivalence relation 
generated by declaring $\F$ equivalent to 
$G: \R^{n+1} \to \R^{n+1}$, with $G(x,t)=(F(tx)/t,t)$. 
In that case, for $t \ne 0$ consideration of the \J matrix
of $G$  shows that $j(G)(x,t)=j(F)(tx)$, a result that then 
also holds for $t=0$, by continuity.
So Segre equivalence preserves nondegeneracy, 
nonsingularity, and the Keller property. Again all preservation properties apply in both directions. 
It also preserves \p maps and \edr maps, 
because $G(x,1)=(F(x),1)$.  
For $t \ne 0$, the set $G^{-1}(y,t)=
\{(x/t,t)|x \in F^{-1}(ty)\}$ for any $y \in \R^n$.  
That implies that 
injectivity and surjectivity are preserved 
provided that $G$ is bijective on the set of 
points $(x,0)$, a condition equivalent to $j(F)(0) \ne 0$. 
In particular, for bijective nonsingular $F$ 
one has $G^{-1}(y,t)=(F^{-1}(ty)/t,t)$, 
and so \p and \edr inverses are also preserved. 

\begin{Thm}
Assume $\F$ and $\G$ are real analytic maps 
sending $0$ to $0$ 
and that they are Segre equivalent.
 If either one is an \edr nonsingular 
map, then so is the other, and  each of the numerical
attributes dex, sag, and mfs has the same value
for both maps. 
\end{Thm}

\begin{proof}
Only the equality of the attributes needs checking, 
and only for the case $G(x,t)=(F(tx)/t,t)$. 
Fibers over points $(y,0)$ are of size $1$ 
and for $t \ne 0$ the fiber of $G$ over $(y,t)$ has
the same size as the fiber of $F$ over $ty$ ,
by the formula given above for the set $G^{-1}(y,t)$. 
So mfs is preserved. 

Now consider the automorphism of the field $\R(X,t)$
that sends $x_i$ to $tx_i$ and $t$ to itself.
It restricts to an isomorphism of $\R(F,t)$ onto $\R(G)$.

This is just an instance of special case ii) in the proof of the previous theorem. 
So dex  and sag have the same value for $G$ as for  
$(F,t)$, and hence, by the previous theorem, as for $F$.
\end{proof}

The coimage of a map $\F$ is the set $\R^n \setminus F(\R^n)$ 
of points in the codomain that are not in the image of $F$.
In the complex JC context, it is well known that the coimage 
has complex codimension at least two. Briefly, the reasoning is as follows.  Since the coimage is closed and constructible, if it has
codimension less than two it  
contains an irreducible hypersurface $h=0$, $h \circ F$ vanishes 
nowhere and so is constant, contradicting the algebraic 
independence of the components of $F$. 
In the SRJC and RRJC contexts, there are no parallel  results for 
the real codimension of the coimage, even if the map has dense 
image. 
The \pk maps, however, do have finite coimages, which are indeed of
codimension two in $\R^2$.
 
 \begin{Thm}
Assume $\F$ and $\G$ are everywhere defined rational 
nonsingular maps and that they are birationally stably or Segre equivalent.
Then  the codimension of the coimage is the same for both maps. 
\end{Thm}

\begin{proof}
Only the Segre case is not totally trivial, and  in the base Segre case 
the coimage of $G$ consists of the point $(a/t,t)$ for $a$ in the 
coimage of $F$ and $t \ne 0$. 
\end{proof}

%% file: gorzamp.tex
Let $f_i=x_i+l_i^3$ be the components of a map $F$ 
of cubic linear type in dimension $n$. 
It is customary to write $F$ 
in the compact form $F(x)=x+(Ax)^{*3}$, 
where $A$ is the matrix of coefficients of the linear 
forms and the exponent indicates componentwise cubing. 
Let $G$ be a map of cubic homogeneous type in dimension $m<n$. 
A GZ pairing between $G$ and $F$ is given by two matrices $B$ and $C$, respectively of sizes 
$m \times n$ and $n \times m$, satisfying $BC=I$, $\text{ker} B = \text{ker} A$, and $G(x)=BF(Cx)$ for all $x \in \R^m$.
The original definition \cite{GZpairs} writes $F$ as
$F(x) = x-(Ax)^{*3}$, 
but the different sign affects only some formulas not used here.

\begin{Thm}
If $G$ and $F$ are GZ paired, then they are polynomially
stably equivalent.
\end{Thm}

\begin{proof}

 Note that $\text{ker} C = 0$, $\text{Im} B = \R^m$, and that
$\R^n$ is the direct sum of $\text{Im} C$ and $\text{ker} B$.
Choose a linear isomorphism $D$ from $\R^{n-m}$ to 
$\text{ker} B$.
Let $E$ be its inverse. Consider the extension of $G$ by
fresh variables to 
$G'=(g_1,\ldots,g_m,z_{m+1},\ldots,z_n)$. 
Let $C'(x,z) = Cx+D(z) \in \R^n$ and $B'(x) =(Bx,E'(x))$,
where $E'$ is the linear extension of $E$ to $\R^n$ that is 
$0$ on $\text{Im} C$.

Note that both $B'$ and $C'$ are linear 
automorphisms of $\R^n$. 
Observe that $F(Cx+D(z))=Cx+D(z)+(ACx)^{*3}$. 
So $(B'\circ F\circ C')(x,z)=(G(x),z+E'((ACx)^{*3}))=
G'\circ (x,z+H(x))$, where $H$ is cubic homogeneous.
Since $(x,z+H(x))$ has the obvious inverse $(x,z-H(x))$,
it follows that $G$ and $F$ are polynomially stably equivalent.
\end{proof}

Remarks. The same reasoning works over any ground field $k$.
There is also nothing special about the use of $3$ as the 
exponent.  
All works just as well for power homogeneous and 
power linear maps of the same degree $d >1$.

In a GZ pairing  the rank of the map of cubic linear type, 
meaning the rank of the coefficient matrix $A$, 
is the same as the dimension $m$ of 
the map of cubic homogeneous type. 
In \cite{DRrk2} the SRJC is proved for all maps of cubic linear type and rank $2$. 
The heart of the proof is a theorem  proving 
that the SRJC is true 
for all maps of cubic homogeneous type in dimension $2$. 
These facts are of interest when considering 
structured counterexamples in higher dimensions.

Remarks. For reasons not clear to me,  
\cite{DRrk2} presents the results mentioned above for  
maps with an everywhere positive \J determinant, 
which is automatically true for nonsingular 
maps of cubic homogeneous type. Other results include the SRJC for all maps of cubic linear type in dimension $3$. 

The dimension $2$ results were later improved to 
cover \p maps with
components of degree at most $3$ \cite{RJCn=2cubic}, 
and then to \p maps with 
one component of degree at most $3$ \cite{RJCn=2pcubic}.

%% file: final.tex
\begin{Thm}\label{reduce-1}
There is an algorithm that transforms a nondegenerate, 
polynomial map 
$F:\R^n \To \R^n$ into a map $G:\R^m \To \R^m$ of cubic homogeneous type, 
where $m$ is generally much larger 
than $n$, using polynomially stable equivalences 
and a single Segre equivalence.
\end{Thm}

\begin{proof}
In each step below a map $F$ is replaced by a map $G$, 
which becomes the new $F$ for the next step. At each 
step both $F$ and $G$ are nondegenerate, 
since that property is preserved by the equivalences.

Step 1. Lower the degree. 
Suppose $F=(\comps)$. 
$F$ is polynomially stably equivalent to 
$(f_1 - (y+a)(z+b),f_2,\ldots,f_n,y+a,z+b)$, 
where $a,b$ are polynomials that depend only on 
$\vars$. Thus, if a term of $f_1$ has the form 
$ab$, with $\deg(a)>1$ and $\deg(b)>1$, it 
can be removed at the cost of introducing two new 
variables and some terms of degree less than $\deg(ab)$.
Repeating this for terms of maximum degree until 
there are no more maximal degree 
terms of the specified form in any component, one 
finally obtains a polynomial map $G$ (in a generally 
much higher dimension), all of whose terms are of degree 
no more than three.
This is a standard algorithm \cite{BCW82,ArnoBook}.  
There is flexibility in the choice of term to 
remove next, and one can opportunistically remove a 
product $ab$ that is not a single term, making choices to 
reach a cubic map more quickly. This step is a polynomial 
stable equivalence.

Step 2. Normalize. 
 $F$ is now cubic and (still) nondegenerate. 
Let $n$ be the current dimension. 
Choose $x_0 \in \R^n$ with $j(F)(x_0) \ne 0$. 
After suitable translations, 
$(J(F)(x_0))^{-1}F$ becomes a cubic map $G$, such  
that $G(0)=0$ and $G'(0)=J(G)(0)$ is the identity 
matrix $I$. This step is an affine (in the vector space sense) 
equivalence.

Step 3. Segre equivalence.
Now $F=X+Q+C$, where $Q$ and $C$ are, respectively, 
 the quadratic and cubic homogeneous components of $F$. 
Let $t$ be a new variable, and put 
$G=(X+tQ+t^2C,t)$. 
This is   
a \p Segre equivalence, as defined previously.

Step 4. Final step. 
Now $F=(X+tQ+t^2C,t)$, with $Q$ quadratic 
homogeneous and $C$ cubic homogeneous, 
and both independent of $t$. 
Define  two polynomial automorphisms
$A_1,A_2$ in $X,Y,t$, where $Y$ is 
a sequence of $n$ additional variables, 
by $A_1=(X-t^2Y,Y,t)$ and $A_2=(X,Y+C,t)$. 
Then $G=A_1 \circ (X+tQ+t^2C,Y,t) \circ A_2$ is 
the map of cubic homogeneous type 
$(X-t^2Y+tQ,Y+C,t)$. 
This step is a polynomial stable equivalence. 
\end{proof}

 The theorem and proof are valid over $\C$ as well as
over $\R$, and, indeed, 
more generally for Keller maps.  All proofs of reduction to
 cubic homogeneous type
start with reduction to degree $3$, 
followed by elimination of the quadratic terms.
The given proof most closely follows that of 
\Druz in \cite{Effective}, which explicitly 
allows for nonconstant \J determinants. 

The main point of the given proof is that for nonsingular \p maps, 
by the preservation results previously proved, the reduction preserves (in both directions)  
not only bijectivity, but also dex, mfs, sag, 
 the Keller property, and  the codimension of the coimage. 

So if it is applied to a \pk map, it  
yields a Yagzhev map $G$, for which $j(G)$ is not constant
and $J(G)$ is not unipotent. 	
Up to inessential details, Hubbers follows the above steps 
in the first part of his 1999 reduction, 
obtaining a cubic map in dimension $n=101$ 
and then a map of cubic homogeneous type 
in dimension $2n+1=203$. 
Hubbers' Yagzhev map in dimension $203$ is thus not Keller 
and satisfies dex = 6, mfs = 2, sag = 1, and 
has a coimage of codimension $2$.

A further reduction to a map of cubic linear type can
be effected using the method of \Druz in \cite{Effective} or the method of GZ pairing developed by Gianluca Gorni and 
Gaetano Zampieri in \cite{GZpairs}.
Since 
GZ pairing has been shown to be 
a polynomial stable equivalence, 
Hubbers' final \Druz map in dimension $1999$ has 
the same properties as those stated for 
his Yagzhev map in dimension $203$.

\begin{Thm}\label{reduce-2}
Any nonsingular $\mathcal{C}^2$ map $F:\R^n \To \R^n$,
is stably equivalent to a $\mathcal{C}^1$ 
nonsingular map $G:\R^{2n} \To \R^{2n}$  
with a symmetric \J matrix. 
The equivalence is birationally stable if $F$ is 
everywhere defined rational, and polynomially 
stable if $F$ is a \p Keller map.
\end{Thm}

\begin{proof}
This is Meng's trick with a reordering of the variables. 
Let $z=(x,v)$ be any point of $\R^{2n}$, 
with $x=(\vars)$ and $v=(x_{n+1},\ldots,x_{2n})$.
Let $h=x\cdot F(v)$,  where the dot denotes the standard inner product of $n$-vectors. 
Let $G$ be the gradient of the scalar function $h$.  
Then $G$ has a symmetric \J matrix 
and $G=(F(v),x\cdot J(F)(v))$, with the dot now  
denoting a vector matrix product and $J(F)$ 
the \J matrix of $F$.  
But 
$(F(v),x\cdot J(F)(v))=(F(x),v\cdot J(F)(x)) \circ (v,x)$
and $(F(x),v\cdot J(F)(x))
=(F(x),v)\circ (x,v\cdot J(F)(x))$.  
 Since $(x,v\cdot J(F)(x))$ has the $\mathcal{C}^1$ inverse 
$(x,v\cdot J(F)^{-1}(x))$, the composition 
$A=(x,v\cdot J(F)(x))\circ (v,x)$ is a 
$\mathcal{C}^1$ automorphism. 
Moreover, if $F$ is 
an \edr or \p Keller map, it is clear that $A$ has the claimed 
properties.
\end{proof}
 
This theorem reduces the entire RRJC, not just the SRJC, to the case of a symmetric \J matrix 
and preserves dex, sag, and mfs. 

It is natural to attempt to combine  
the two main results by 
applying  Theorem \ref{reduce-2} to a Yagzhev map. 
The resulting map is polynomial, 
with only its linear and cubic homogeneous components nonzero.
But its linear part is not the identity.

%% file: ack.tex
Thanks especially to Engelbert Hubbers for providing complete details 
on his reduction procedure \cite{pinchuk99}. 
Michiel deBondt helped simplify the 
last step of the proof 
of Theorem \ref{reduce-1}. 
And both he and Gianluca Gorni sent me proofs that 
GZ pairing preserves mfs, when I 
first raised the question.